\documentclass{amsart}

\usepackage{amsmath,amssymb,amsthm}
\usepackage{hyperref}
\usepackage{xcolor}

\newtheorem{thm}{Theorem}[section]
\newtheorem{prop}{Proposition}[subsection]
\newtheorem{lem}[prop]{Lemma}

\theoremstyle{definition}

\newtheorem{rem}[prop]{Remark}
\newtheorem{exwith}[prop]{Example}

\newtheorem*{ack}{Acknowledgement}

\def\co{\colon\thinspace}

\newcommand{\rmd}{\mathrm d}

\newcommand{\N}{\mathbb N}

\newcommand{\R}{\mathbb R}

\newcommand{\Z}{\mathbb Z}

\newcommand{\lra}{\longrightarrow}
\newcommand{\ra}{\rightarrow}

\DeclareMathOperator{\dR}{\mathrm{dR}}

\DeclareMathOperator{\inj}{\mathrm{inj}}
\DeclareMathOperator{\Int}{\mathrm{Int}}

\begin{document}

\author{Kevin Wiegand}
\author{Kai Zehmisch}
\address{Mathematisches Institut, Westf\"alische Wilhelms-Universit\"at M\"unster, Einsteinstr. 62,
D-48149 M\"unster, Germany}
\email{kevin.wiegand@uni-muenster.de, kai.zehmisch@uni-muenster.de}

\title[Virtually contact structures]
{Two constructions of virtually contact structures}

\date{23.05.2017}

\begin{abstract}
  Motivated by recent developments
  in proving the Weinstein conjecture
  we introduce the notion of covering
  contact connected sum
  for virtually contact manifolds
  and construct virtually contact structures
  on boundaries of subcritical handle bodies.
\end{abstract}

\thanks{KW is supported by the SFB 878 - Groups, Geometry and Actions.
KZ is partially supported by the SFB/TRR 191 - Symplectic Structures in Geometry,
Algebra and Dynamics and DFG grant ZE 992/1-1.
}

\maketitle

%%%%%%%%%%%%%%%%%%%%%%%%%%%%%%%%%%%
%%%%%%%%%%%%%%%%%%%%%%%%%%%%%%%%%%%

\section{Introduction\label{sec:intro}}

Virtually contact structures
naturally appear in classical mechanics
in the study of magnetic flows
on compact Riemannian manifolds $(Q,h)$
of negative sectional curvature.
The appearance of the magnetic $2$-form
$\sigma$ on $Q$ is reflected in the use of
the twisted symplectic form on $T^*Q$
obtained by adding the pull back of $\sigma$
along the cotangent bundle projection to
$\rmd\mathbf{p}\wedge\rmd\mathbf{q}$.
As it turns out, energy surfaces $M\subset T^*Q$
of twisted cotangent bundles
need not to be of contact type in general.

It was pointed out by
Cieliebak--Frauenfelder--Parternain \cite{cfp10}
that in many interesting cases
a certain covering $\pi\co M'\ra M$
of the energy surface $M\subset T^*Q$
admits a contact form $\alpha$
whose Reeb flow projects to the
Hamiltonian flow on the energy surface
$M\subset T^*Q$ up to parametrization.
Moreover, the contact form $\alpha$ admits
uniform upper and lower bounds
with respect to a lifted metric.
In this situation, the manifold $M$ together with
the odd-dimensional symplectic form $\omega$
obtained by restriction of the twisted symplectic form
to $TM$ is called a virtually contact manifold.
In particular,
questions about periodic orbits
on virtually contact manifolds $(M,\omega)$
can be answered on the covering space $M'$
with help of the contact form $\alpha$.

If the covering space $M'$
of a virtually contact manifold $(M,\omega)$
is compact, and hence the covering $\pi$ is finite,
the energy surface $M$ will be of contact type.
The existence question
about periodic orbits in this case is subject
to the Weinstein conjecture,
see \cite{wein79},
and 
the virtually contact manifold $(M,\omega)$
is called to be trivial.
If the covering $\pi$ is infinite
with a non-amenable covering group,
one is intended to study
periodic orbits on a non-compact contact manifold
$(M',\alpha)$.
This is because the covered energy surface $M$
is not necessarily of contact type.

In general, open contact manifolds do admit
aperiodic Reeb flows as the standard contact form
$\rmd z+\mathbf{y}\rmd\mathbf{x}$
on Euclidean spaces shows.
In order to achieve existence of periodic Reeb orbits
additional conditions are required,
cf.\ \cite{bae14, bpv09, bprv13, cfp10, dpp15, sz16}.
It was asked by G.\ P.\ Paternain
whether virtually contact manifolds have to admit
periodic orbits.
The question was answered positively in many instances
by  Cieliebak--Frauenfelder--Parternain \cite{cfp10}
and, more recently,
by Bae--Wiegand--Zehmisch \cite{bwz}.
The content of the following theorem is
to give a large class of examples
to which the existence theory
developed in \cite{bwz} applies.

\begin{thm}
\label{thm:mainthm}
  For all $n\geq2$
  there exist non-trivial closed virtually contact manifolds $M$
  of dimension $2n-1$ which topologically
  are a connected sum such that the corresponding
  belt sphere represents a non-trivial homotopy class
  in $\pi_{2n-2}M$.
  The involved covering space $M'$
  is obtained by covering contact connected sum.
\end{thm}

The virtually contact structures studied by
Cieliebak--Frauenfelder--Parternain \cite{cfp10}
are diffeomorphic to unit cotangent bundles
of negatively curved manifolds.
The examples we are going to construct
in Section \ref{subsec:pfofthm} are obtained
by covering connected sum,
which is an extension
of the contact connected sum operation
to the class of virtually contact manifolds.
In Section \ref{subsec:st*qisprim} we will show
that unit cotangent bundles of aspherical manifolds
are prime.
This implies that the covering connected sum
produces virtually contact structure
that differ from those studied in \cite{cfp10}.

Motivated by Hofer's \cite{hof93}
verification of the Weinstein conjecture
for closed overtwisted contact $3$-manifolds
Bae \cite{bae14} constructed
virtually contact manifolds in dimension $3$
using a covering version of the Lutz twist.
The topology of
the base manifold of the covering
thereby stays unchanged.
The total space of the resulting covering
is an overtwisted contact manifold
and the virtually contact structure will be
non-trivial.
In Proposition \ref{prop:pertofcontwnegcurv}
we present a tool to produce more
examples of that nature.
Let us remind that
non-trivially here and in Theorem \ref{thm:mainthm}
means that the symplectic form
on the odd-dimensional manifold
is not the differential of a contact form.

The verification of the Weinstein conjecture by Hofer \cite{hof93}
for closed reducible $3$-manifolds suggests
the question about the existence of
non-trivial virtually contact $3$-manifolds
with non-vanishing $\pi_2$.
This question is answered by Theorem \ref{thm:mainthm}.
In fact, the results in
\cite{grz14,grz16,gz16a,gz16b,gnw16}
motivated the definition of the
covering contact connected sum.
Extending the work of Geiges--Zehmisch \cite{gz16a}
the existence of periodic orbits
for virtually contact structures
addressed by Theorem \ref{thm:mainthm}
that in addition admit a $C^3$-bounded contact form
on the total space of the covering
is shown in \cite{bwz}. 

In Section \ref{sec:morsepot}
we will give a second construction
of virtually contact structures
that will be obtained via energy surfaces
of classical Hamiltonian functions
in twisted cotangent bundles.
The corresponding energy will be below
the Ma\~n\'e critical value of the
involved magnetic system
so that the energy surfaces intersect
the zero section of the cotangent bundle.
The topology of the energy surface
is determined by the potential function
on the configuration space
according to Morse theoretical considerations.

\begin{thm}
\label{thm:2ndthm}
  For any $n\geq2$ and given $b\in\N$
  there exists a closed virtually contact manifold $M$
  of dimension $2n-1$
  such that $\pi_nM$ and the image in $H_nM$
  under the Hurewicz homomorphism, resp.,
  contain a subgroup
  isomorphic to $\Z^b$.
  The virtually contact manifold $M$
  appears as the energy surface
  of a classical Hamiltonian function
  in a twisted cotangent bundle $T^*Q$.
  The rank $b$ of the subgroup $\Z^b$
  is the first Betti number
  of the configuration space $Q$.  
  If $n\geq3$
  the virtually contact structure on $M$
  is non-trivial.
\end{thm}

Based on the work of
Ghiggini--Niederkr{\"u}ger--Wendl \cite{gnw16}
existence of periodic solutions in the context of
Theorem \ref{thm:2ndthm} can be shown
provided that the magnetic form
has a $C^3$-bounded primitive
on the universal cover of $Q$,
see \cite[Theorem 1.1 and 1.2]{bwz}.
Furthermore, by the classification
obtained by Barth--Geiges--Zehmisch
in \cite[Theorem 1.2.(a)]{bgz}
the contact structure on $M$ obtained by
homotoping the magnetic term
of the twisted cotangent bundle $T^*Q$ to zero
is different from the standard contact structure
on the unit cotangent bundle $ST^*P$
of any Riemannian manifold $P$.

%%%%%%%%%%%%%%%%%%%%%%%%%%%%%%%%%%%
%%%%%%%%%%%%%%%%%%%%%%%%%%%%%%%%%%%

\section{A construction via surgery\label{sec:construction}}

%%%%%%%%%%%%%%%%%%%%%%%%%%%%%%%%%%%

\subsection{Definitions\label{subsec:definitions}}

The following terminology was introduced in \cite{bae14, cfp10}.
Let $M$ be a $(2n-1)$-dimensional manifold
for $n\geq2$.
A closed $2$-form $\omega$ on $M$ is called
{\bf symplectic} if $\ker\omega$
is a $1$-dimensional distribution.
The pair $(M,\omega)$ is an
{\bf odd-dimensional symplectic manifold}.
It is called {\bf virtually contact} if the following
two conditions are satisfied:

{\bf Primitive:}
There exist a covering $\pi\co M'\ra M$
and a contact form $\alpha$ on $M'$
such that $\pi^*\omega=\rmd\alpha$,
so that $\alpha$ is a primitive of the lift of $\omega$
and $\alpha$ defines a contact structure $\xi=\ker\alpha$
on the covering space $M'$.

{\bf Bounded geometry:}
There exist a metric $g$ of bounded geometry on $M$ and
a constant $c>0$ subject to the following geometric bounds:
\[
\sup_{M'}|\alpha|_{(\pi^*g)^{\flat}}<\infty
\label{eq:gb1}\tag{gb$_1$}
\]
with respect to the dual of the pull back metric
$\pi^*g$;
and for all $v\in\ker\rmd\alpha$
\[
|\alpha(v)|>c|v|_{\pi^*g}\,.
\label{eq:gb2}\tag{gb$_2$}
\]
If the manifold $M$ is closed
any metric $g$ will be of {\bf bounded geometry},
i.e.\ the injectivity radius $\inj_g>0$ of $(M,g)$
is positive and the absolut value of the sectional curvature
$|\!\sec_g\!|$ is bounded.

The tuple
\[
\big(\pi\co M'\ra M,\alpha,\omega,g\big)
\]
is called {\bf virtually contact structure}
and $(M,\omega)$ a {\bf virtually contact manifold}.
A virtually contact manifold is {\bf non-trivial}
if $\omega$ is not the differential
of a contact form on $M$.
In particular, the covering $\pi$
of a non-trivial virtually contact structure is infinite
and $M$ has a non-amenable fundamental group.
A virtually contact structure
is called {\bf somewhere contact}
if there exist an open subset $U$ of $M$
and a contact form $\alpha_U$ on $U$
such that $\pi^*\alpha_U=\alpha$
on $\pi^{-1}(U)$.

%%%%%%%%%%%%%%%%%%%%%%%%%%%%%%%%%%%

\subsection{Covering connected sum\label{subsec:covconsum}}

For $i=1,2$ we consider two somewhere contact virtually contact structures
$\big(\pi_i\co M'_i\ra M_i,\alpha_i,\omega_i,g_i\big)$.
Denote by $U_i$, $i=1,2$, an open subset of $M_i$
on which a contact form $\alpha_{U_i}$ exists
according the the definition of being somewhere contact.
Given a bijection $b$ between the fibers
of the coverings $\pi_1$ and $\pi_2$
over the respective base points of  $M_1$ and $M_2$
we define a covering connected sum as follows:

Let $D_i^{2n-1}$, $i=1,2$, be a closed embedded disc
contained in $U_i$ such that a neighbourhood
of the disc is equipped with Darboux coordinates
for the contact form $\alpha_{U_i}$.
We perform contact index-$1$
surgery as described in \cite{gei08}
identifying $\partial D_i^{2n-1}$
with the boundary $\{i\}\times S^{2n-2}$
of the upper boundary of $[1,2]\times S^{2n-2}$
the $1$-handle $[1,2]\times D^{2n-1}$.
The resulting contact form on
the connected sum $U_1\#U_2$
is denoted by $\alpha_{U_1}\#\alpha_{U_2}$.
Notice, that $\alpha_{U_1}\#\alpha_{U_2}$
coincides with $\alpha_{U_i}$ on $U_i\setminus D_i^{2n-1}$.
Let $\omega$ be the odd-dimensional symplectic form
on $M_1\#M_2$ that coincides with
$\omega_i$ on $M_i\setminus U_i$
for $i=1,2$ and with $\rmd(\alpha_{U_1}\#\alpha_{U_2})$
on $U_1\#U_2$.
Similarly, a metric $g$ of bounded geometry
can be defined via extension of $g_1$ and $g_2$
over the handle part.

In order to define a connected sum
of the coverings $\pi_i$ we may assume that
the base point $x_i$ of $M_i$
lies on the boundary of $D_i^{2n-1}$.
Moreover,
we choose the subset $U_i$, $i=1,2$,
so small such that
$\pi_i^{-1}(U_i)$ decomposes into a
disjoint union of open sets $U_i^y$, $y\in\pi_i^{-1}(x_i)$,
and that the restriction of $\pi_i$ to $U_i^y$
is an embedding into $M_i$ for all $y\in\pi_i^{-1}(x_i)$.
Then, topologically,
we define a family of connected sums
$U_1^y\#U_2^{b(y)}$ according to the bijection
$b$ between the fibers over the base points.

The restrictions of the contact forms $\alpha_i|_{U_i^y}$
correspond to the local contact form $\alpha_{U_i}$
diffeomorphically via $\pi_i$, $i=1,2$.
A contact form on
\[
U_1^y\#U_2^{b(y)}
\]
can be defined equivariantly
via contact connected sum as follows:
Let $M'_1\#_bM'_2$ be the manifold
obtained by gluing
$M'_i\setminus\pi_i^{-1}(U_i)$, $i=1,2$,
with $U_1^y\#U_2^{b(y)}$, $y\in\pi_i^{-1}(x_1)$,
along their boundaries in the obvious way.
We obtain a covering
\[
\pi\co M'_1\#_bM'_2\lra M_1\#M_2
\]
that restricts to $\pi_i$ on $M'_i\setminus\pi_i^{-1}(U_i)$, $i=1,2$,
and defines the trivial covering over the handle parts
being the identity restricted to each of the sheets.
Then $M'_1\#_bM'_2$ carries a contact form $\alpha$
whose restriction to the union of the $U_1^y\#U_2^{b(y)}$, $y\in\pi_1^{-1}(x_1)$,
coincides with $\pi^*\big(\alpha_{U_1}\#\alpha_{U_2}\big)$
and that restricts to $\alpha_i$ on
$M'_i\setminus\pi_i^{-1}(U_i)$, $i=1,2$.
Because each of the involved handles is compact
the covering $\pi\co M'\ra M$ of $M=M_1\#M_2$
by $M'=M'_1\#_bM'_2$ defines a virtually contact structure
given by $\big(\pi\co M'\ra M,\alpha,\omega,g\big)$.

\begin{rem}
 \label{rem:conthandle}
 Observe, that the modle contact handle
 used for the contact connected sum
 carries obvious periodic characteristics of
 $\ker\big(\rmd(\alpha_{U_1}\#\alpha_{U_2})\big)$
 inside the {\bf belt sphere} $\{3/2\}\times S^{2n-2}$.
 The situation changes after a perturbation
 of $\alpha_{U_1}\#\alpha_{U_2}$
 obtained by a multiplication with a positive function
 that is constantly equal to $1$ in the complement
 of the handle.
 This operation changes the virtually contact structure
 on the connected sum $M=M_1\#M_2$
 but not the contact structure $\xi=\ker\alpha$ on the covering $M'$.
 Still, there exists a contact embedding
 of the modle contact handle into $(M',\xi)$.
\end{rem}

\begin{lem}
\label{lem:nontrivaftercovconsum}
 For $i=1,2$ let $\big(\pi_i\co M'_i\ra M_i,\alpha_i,\omega_i,g_i\big)$
 be a somewhere contact virtually contact structure.
 If $\omega_1$ is non-exact,
 then the odd-dimensional symplectic form $\omega$
 on $M_1\#M_2$ corresponding to
 the virtually contact structure
 \[
 \big(\pi\co M'\ra M,\alpha,\omega,g\big)
 \]
 obtained by covering contact connected sum
 is non-exact.
\end{lem}

\begin{proof}
 We argue by contradiction and
 continue the use of notation from above.
 Suppose that the symplectic form $\omega$
 on the $(2n-1)$-dimensional connected sum
 $M=M_1\#M_2$ has a primitive.
 Then the restriction $\omega_1$ of $\omega$
 to $M_1\setminus D^{2n-1}_1$ does.
 An interpolation argument
 for primitives in terms of Mayer--Vietoris sequence
 in de Rham cohomology 
 using $H^1_{\dR}(S^{2n-2})=0$
 shows that the odd-dimensional symplectic form
 $\omega_1$ on $M_1$ has a primitive.
 
 A more elementary argument goes as follows:
 Denote the primitive of the restriction
 of $\omega_1$ to $M_1\setminus D^{2n-1}_1$
 by $\lambda$.
 Observe that $\lambda|_U$ is a closed $1$-form
 and, hence, exact in a neighbourhood $D'$
 of the disc $D^{2n-1}_1$.
 Cutting a primitive function of $\lambda|_{D'}$
 down to zero in radial direction
 we can assume that $\lambda$ vanishes near
 $\partial D^{2n-1}_1\subset U\subset M$.
 In other words,
 a perturbation of $\lambda$
 extends over $D^{2n-1}_1$ by zero
 resulting in a primitive of $\omega_1$.
 This is a contradiction.
\end{proof}

%%%%%%%%%%%%%%%%%%%%%%%%%%%%%%%%%%%

\subsection{Magnetic flows\label{subsec:magflow}}

Virtually contact structures appear naturally
on energy surfaces of classical Hamiltonians
on twisted cotangent bundles.
We briefly recall the construction following \cite{bz15, cfp10}.

Let $(Q,h)$ be a closed $n$-dimensional
Riemannian manifold and let $\sigma$ be a
closed $2$-form on $Q$, which is called the
{\bf magnetic form}.
The {\bf Liouville form} on the cotangent bundle
$\tau\co T^*Q\ra Q$ is the $1$-form $\lambda$
on the total space $T^*Q$ that is given by
$\lambda_u=u\circ T\tau$ for all covectors $u\in T^*Q$.
The {\bf twisted symplectic form} by definition is
$\omega_{\sigma}=\rmd\lambda+\tau^*\sigma$.
For a smooth function $V$ on $Q$,
the so-called {\bf potential}, and
the dual metric $h^{\flat}$ of $h$ we consider the
Hamiltonian of classical mechanics
\[
H(u)=\frac12|u|^2_{h^{\flat}}+V\big(\tau(u)\big)\,.
\]
For energies $k>\max_{Q}V$ we consider
the energy surfaces $\{H=k\}$,
which are regular and in fact diffeomorphic
to the unit cotangent bundle $ST^*Q$
via a diffeomorphism induced by a fibrewise radial isotopy.

It is of particular interest
whether the Lorentz force induced by
the magnetic $2$-form $\sigma$
comes from a potential $1$-form.
Up to lifting $\sigma$ to a certain cover
this will be the case at least 
for so-called weakly exact $2$-forms:
Denoting by $\mu\co\widetilde{Q}\ra Q$
the universal covering of $Q$
we call the $2$-form $\sigma$ on $Q$ {\bf weakly exact}
if there exists a $1$-form $\theta$ on $\widetilde{Q}$
such that $\mu^*\sigma=\rmd\theta$.
In the following we will assume
that the magnetic form $\sigma$ is weakly exact.
Therefore, it is natural to lift the Hamiltonian system
to the universal cover.

The covering map $\mu$ induces a natural map
$T^*\mu\co T^*\widetilde{Q}\ra T^*Q$
that is given by
\[
\tilde{u}\longmapsto\tilde{u}\circ\big(T\mu_{\tilde{\tau}(\tilde{u})}\big)^{-1}\,,
\]
where $\tilde{\tau}\co T^*\widetilde{Q}\ra\widetilde{Q}$
denotes the cotangent bundle of $\widetilde{Q}$
and $\mu_{\tilde{\tau}(\tilde{u})}$ is the germ
of local diffeomorphism at $\tilde{\tau}(\tilde{u})$
that coincides with $\mu$ near $\tilde{\tau}(\tilde{u})$.
Naturallity can be expressed by saying that
$\mu\circ\tilde{\tau}=\tau\circ T^*\mu$
so that
\[
\big(T^*\mu\big)^*\lambda=\tilde{\lambda}\,,
\]
where $\tilde{\lambda}$ denotes the Liouville form
on $T^*\widetilde{Q}$.
Moreover, $T^*\mu$ itself is a covering,
which because of the homotopy equivalence
$T^*\widetilde{Q}\simeq\widetilde{Q}$
can be used to represent the universal covering
of $T^*Q$.
The lifted Hamiltonian $\widetilde{H}=H\circ T^*\mu$
is a Hamiltonian of classical mechanics
\[
\widetilde{H}(\tilde{u})=
\frac12|\tilde{u}|^2_{(\tilde{h})^{\flat}}+
\widetilde{V}\big(\tilde{\tau}(\tilde{u})\big)\,,
\]
$\tilde{u}\in T^*\widetilde{Q}$,
with respect to the lifted metric $\tilde{h}=\mu^*h$
and the lifted potential energy function
$\widetilde{V}=V\circ\mu$.
The preimage of $\{H=k\}$ under $T^*\mu$
is equal to $\{\widetilde{H}=k\}$.
In fact,
an application of the implicit function theorem
yields that the restriction
\[
\pi=T^*\mu|_{\{\widetilde{H}=k\}}
\]
defines a covering projection
\[
M'=\{\widetilde{H}=k\}
\longrightarrow
\{H=k\}=M\,.
\]
Because there exists a $1$-form $\theta$ on $\widetilde{Q}$
such that $\mu^*\sigma=\rmd\theta$ we find that
\[
\big(T^*\mu\big)^*\tau^*\sigma=\rmd (\tilde{\tau}^*\theta)\,,
\]
so that
\[
\big(T^*\mu\big)^*\omega_{\sigma}=
\rmd\tilde{\lambda}+\tilde{\tau}^*\rmd\theta=:
\tilde{\omega}_{\rmd\theta}
\]
has primitive $\tilde{\lambda}+\tilde{\tau}^*\theta$.
The restriction to $TM'$ is denoted by
\[
\alpha=\big(\tilde{\lambda}+\tilde{\tau}^*\theta\big)|_{TM'}\,.
\]
Setting $\omega=\omega_{\sigma}|_{TM}$
we obtain a map
\[
\pi\co\big(M',\rmd\alpha\big)\longrightarrow\big(M,\omega\big)
\]
of odd-dimensional symplectic manifolds.
The question that we will address in the following is
under which conditions the $1$-form $\alpha$
will be a contact form on $M'$.

\begin{rem}
 \label{rem:topofpi}
 The topology of the covering $\pi$
 can be determined as follows.
 By the choice $k>\max_QV$
 the covering space $M'$ is diffeomorphic to
 $ST^*\widetilde{Q}$
 so that $M'$ carries the structure
 of a $S^{n-1}$-bundle over $\widetilde{Q}$.
 The long exact sequence of the induced Serre fibration
 shows that the inclusion $S^{n-1}\ra M'$
 of the typical fibre yields a
 surjection of fundamental groups.
 Therefore, if $Q$ is not a surface, i.e.\ $n>2$,
 then $M'$ is simply connected and
 $\pi$ the universal covering.
 If $Q$ is a surface,
 then in view of uniformization
 $\pi$ is a covering of $M=ST^*Q$
 with covering space $M'$ equal to
 $\R^2\times S^1$ for $Q\neq S^2$;
 otherwise, if $Q=S^2$, then $\pi$ is the trivial one-sheeted
 covering of $\R P^3$.
\end{rem}

%%%%%%%%%%%%%%%%%%%%%%%%%%%%%%%%%%%

\subsection{Bounded primitive\label{subsec:boundprim}}

We assume that the primitive $\theta$ of $\mu^*\sigma$,
viewed as a section $\widetilde{Q}\ra T^*\widetilde{Q}$
of $\tilde{\tau}$, is {\bf bounded} with respect to
the lifted metric $\tilde{h}$,
i.e.
\[
\sup_{\widetilde{Q}}|\theta|_{(\tilde{h})^{\flat}}
<\infty\,.
\]
This will be the case
for negatively curved Riemannian manifolds $(Q,h)$
as it was pointed out by Gromov \cite{gr91},
see Example \ref{ex:gromovsexample} below.
By compactness of $Q$ the lifted potential $\widetilde{V}$
is bounded on $\widetilde{Q}$ so that the function
$\widetilde{H}\circ\theta\co\widetilde{Q}\ra\R$ is bounded from above,
i.e.
\[
\sup_{\widetilde{Q}}\widetilde{H}(\theta)<\infty\,.
\]
The following proposition is contained in
\cite[Lemma 5.1]{cfp10}.

\begin{prop}
 \label{prop:boundgivescontact}
 We assume the situation described
 in Section \ref{subsec:magflow}.
 Let $g$ be a metric on $M$.
 If $\mu^*\sigma$ has a bounded primitive $\theta$,
 then for all $k>\sup_{\widetilde{Q}}\widetilde{H}(\theta)$
 the tuple $\big(\pi\co M'\ra M,\alpha,\omega,g\big)$
 is a virtually contact structure.
 The odd-dimensional symplectic form $\omega$
 of the virtually contact structure is non-exact
 provided $\dim Q\geq3$ and
 the magnetic form $\sigma$ on $Q$ is not exact.
 On closed hyperbolic surfaces $Q$ there exist
 magnetic forms $\sigma$ on $Q$ for which the construction
 yields non-trivial virtually contact structures.
\end{prop}

\begin{proof}
 Choose $k$ such that $k>\sup_{\widetilde{Q}}\widetilde{H}(\theta)$.
 As in \cite[Lemma 5.1]{cfp10}
 we find a $\varepsilon>0$ such that
 \[
 |\theta|_{(\tilde{h})^{\flat}}+
 \varepsilon\leq
 \sqrt{2(k-V)}
 \]
 uniformly on $\widetilde{Q}$.
 Notice, that
 \[
 \big(\tilde{\lambda}+\tilde{\tau}^*\theta\big)
 \big(X_{\widetilde{H}}\big)(\tilde{u})=
 |\tilde{u}|^2_{(\tilde{h})^{\flat}}+
 (\tilde{h})^{\flat}(\tilde{u},\theta)\geq
 |\tilde{u}|_{(\tilde{h})^{\flat}}
 \Big(
   |\tilde{u}|_{(\tilde{h})^{\flat}}-
   |\theta|_{(\tilde{h})^{\flat}}
 \Big)\,,
 \]
 where $X_{\widetilde{H}}$ is the Hamiltonian vector field
 of the Hamiltonian system $(\tilde{\omega}_{\rmd\theta},\widetilde{H})$.
 Because $M'$ is the regular level set
 $\{\widetilde{H}=k\}$
 we get 
 $\alpha\big(X_{\widetilde{H}}\big)\geq\varepsilon^2$
 on $M'$.
 In particular,
 $\alpha$ is a contact form on $M'$,
 see \cite[Chapter 4.3]{hoze94}.
 Because $(\tilde{\omega}_{\rmd\theta},\widetilde{H})$
 is the lift of $(\omega_{\sigma},H)$ via
 $T^*\mu$ we obtain $T(T^*\mu)\big(X_{\widetilde{H}}\big)=X_H$.
 Hence, the restriction of $X_{\widetilde{H}}$
 to $M'$ is bounded for any choice of metric on $M$,
 which by construction is a closed manifold.
 This implies \eqref{eq:gb2}.
 
 In order to verify \eqref{eq:gb1} we choose the metric
 on the total space $T^*\widetilde{Q}$
 induced by the splitting
 into horizontal and vertical distribution
 with respect to the Levi--Civita connection of $\tilde{h}$.
 This induces a metric on $M'$
 and turns $T\tilde{\tau}$ into an orthogonal projection operator,
 whose operator norm is bounded by $1$.
 Hence,
 $\tilde{\tau}^*\theta=\theta_{\tilde{\tau}}\circ T\tilde{\tau}$
 and $\tilde{\lambda}_{\tilde{u}}=\tilde{u}\circ T\tilde{\tau}$
 are uniformly bounded because $\theta$ and
 $\frac12|\tilde{u}|^2_{(\tilde{h})^{\flat}}=
 k-\widetilde{V}\big(\tilde{\tau}(\tilde{u})\big)$
 are.
 This shows that the contact form $\alpha$ is bounded.
 
 Therefore,
 $\big(\pi\co M'\ra M,\alpha,\omega,g\big)$
 is a virtually contact structure.
 It remains to show that the virtually contact structure
 has a non-exact odd-dimensional symplectic form
 provided that $n\geq3$ and $\sigma$ is not exact.
 Observe that as in Remark \ref{rem:topofpi}
 one verifies that $M$ is an $S^{n-1}$-bundle over $Q$.
 The Gysin sequence yields an injection
 $(\tau|_M)^*$ from the second de Rham cohomology of $Q$
 into the one of $M$.
 Hence,
 $\tau^*\sigma|_{TM}$ is non-exact too
 so that the restriction $\omega$
 of the twisted symplectic form $\omega_{\sigma}$
 to $TM$ is non-exact.
 This shows non-exactness of the symplectic form
 of the resulting
 virtually contact structures for $n\geq3$.
 
 We discuss non-triviality of the virtually contact structure
 for $n=2$.
 Only closed orientable surfaces $Q$
 admit non-exact $2$-forms.
 By the Gysin sequence
 the $2$-form $\tau^*\sigma|_{TM}$
 is non-exact only for the $2$-torus.
 The argumentation from \cite[Example 0.1.A]{gr91}
 shows that any primitive of $\mu^*\sigma$
 on the cover $\R^2$ is unbounded
 and, therefore, can not result
 into a virtually contact structure.
 This excludes the case that $Q$ is a torus.
 By Remark \ref{rem:topofpi}
 we also can ignore the case $Q$ being $S^2$.
 For the remaining hyperbolic surfaces
 it was shown in
 \cite[Theorem B.1]{con06}
 that there are induced
 virtually contact structures
 $\big(\pi\co M'\ra M,\alpha,\omega,g\big)$
 that are non-trivial,
 cf.\ \cite[p.\ 1833, (ii)]{cfp10}
 and \cite[Chapter 4.3]{hoze94}.
 We remark that examples of contact type are constructed in
 \cite{gin96}.
\end{proof}

\begin{exwith}
 \label{ex:gromovsexample}
 Let $(Q,h)$ be a closed Riemannian manifold
 of negative sectional curvature and let $\sigma$
 be a closed $2$-form on $Q$.
 Then the lift $\mu^*\sigma$
 along the universal covering $\mu\co\widetilde{Q}\ra Q$
 has a bounded primitive $\theta$ on $(\widetilde{Q},\tilde{h})$,
 see \cite[0.2.A.]{gr91}.
 We remark that by the theorem of Hadamard--Cartan
 $\widetilde{Q}$ is diffeomorphic to $\R^n$
 so that $M'=\R^n\times S^{n-1}$
 and $Q$ is an aspherical manifold.
 
 By Preissmann's theorem
 the product $Q_1\times Q_2$
 of two negatively curved manifolds
 does not admit a metric of
 negative sectional curvature.
 But still such a product $Q_1\times Q_2$
 is aspherical and any closed $2$-form
 of the form
 $\sigma_1\oplus\sigma_2$
 has a bounded primitive
 on the universal cover of $Q_1\times Q_2$.
 
 For more examples the reader is referred to \cite{ked09}.
\end{exwith}

%%%%%%%%%%%%%%%%%%%%%%%%%%%%%%%%%%%

\subsection{Somewhere contact\label{subsec:somecont}}

We will use Proposition \ref{prop:boundgivescontact}
for a construction of somewhere contact
virtually contact structures.
The main observation for that is
that if the magnetic term $\sigma$ vanishes,
then the restriction of $\lambda$ to $TM$
defines a contact form on $M=\{H=k\}$
for all $k>\max_QV$.
Indeed, for $\varepsilon>0$ and $u\in M$ satisfying
$\frac12\varepsilon^2\leq k-V\big(\tau(u)\big)$
we get
\[
\lambda\big(X_H\big)(u)=
|u|^2_{h^{\flat}}\geq
\varepsilon^2
\]
so that \cite[Chapter 4.3]{hoze94} applies.
The same holds true for the Hamiltonian system
that is obtained via a lift along $\mu$,
or if $Q$ is replaced by
a relatively compact open subset $U$ of $Q$.

We consider a closed $2$-form $\sigma$ on $Q$
such that $\{\sigma=0\}$ contains a
non-empty relatively compact open subset $U$.
If the lift of $\sigma$ along $\mu$
has a bounded primitive $\theta$
that vanishes on $\mu^{-1}(U)$,
then the resulting virtually contact structure
that is described in Proposition \ref{prop:boundgivescontact}
will be somewhere contact.
Indeed,
the restriction of the contact form
\[
\alpha=\big(\tilde{\lambda}+\tilde{\tau}^*\theta\big)|_{TM'}
\]
to $M'\cap\big(T^*\mu\big)^{-1}(T^*U)$
equals the one of $\tilde{\lambda}|_{TM'}$,
which is mapped to the contact form
\[
\lambda|_{T\big(M\cap T^*U\big)}
\]
via $\pi=T^*\mu|_{M'}$.

\begin{lem}
 \label{lem:boundsomcont}
 Let $\sigma$ be a closed $2$-form on $Q$
 und $V$ be a non-empty relatively compact open
 subset of $Q$ such that $\sigma|_V=0$.
 Let $\theta$ be a bounded primitive of $\mu^*\sigma$.
 Then there exist an open subset $U\subset\bar{U}\subset V$ of $Q$
 and a bounded primitive
 $\hat{\theta}$ of $\mu^*\sigma$
 that coincides with $\theta$ on the complement of $\mu^{-1}(V)$
 and vanishes on $\mu^{-1}(U)$
 such that the virtually contact structure
 \[
 \Big(\pi\co M'\ra M,
 \alpha=\big(\tilde{\lambda}+\tilde{\tau}^*\hat{\theta}\big)|_{TM'},
 \omega,g\Big)
 \]
 obtained
 in Proposition \ref{prop:boundgivescontact}
 is somewhere contact for all
 $k>\sup_{\widetilde{Q}}\widetilde{H}(\hat{\theta})$.
 The odd-dimensional symplectic form $\omega$
 of the virtually contact structure is non-exact
 provided $\dim Q\geq3$ and
 the magnetic form $\sigma$ on $Q$ is not exact.
\end{lem}

\begin{proof}
 In view of the preceding remarks
 it is enough to show that $\mu^*\sigma$
 has a bounded primitive
 that vanishes on $\mu^{-1}(U)$
 for an open subset $U$ of $Q$.
 In order to do so
 we will assume that $\sigma$ vanishes on an
 embedded closed disc $D^n\cong V\subset Q$.
 The open set $U$ is taken to be
 the Euclidean ball $B_{1/2}(0)$ inside $D^n$.
 Additionally,
 we choose $V$ so small that $\mu^{-1}(V)$
 decomposes into a disjoint union of subsets
 $V^p$ of the universal cover of $Q$
 where the union is taken over all $p\in\mu^{-1}(q)$,
 $q\equiv0$, so that
 \[
 \mu^p:=\mu|_{V^p}\co V^p\lra V
 \]
 is a diffeomorphism for all $p$.
 In a similar way the preimage of $U$ is decomposed into
 sets denoted by $U^p$.
 Taking the metric $\tilde{h}=\mu^*h$ on $\widetilde{Q}$
 the maps $\mu^p$ are in fact isometries.
 
 Consider the given bounded primitive
 $\theta$ of $\mu^*\sigma$
 and denote the restriction of $\theta$
 to $V^p$ by $\theta^p:=\theta|_{V^p}$.
 Notice, that $\rmd\theta^p=0$ for all $p$.
 By the Poincar\'e--Lemma there exists a function
 $f^p\co V^p\ra\R$ such that $\rmd f^p=\theta^p$.
 Choose a cut-off function $\chi$ on $Q$
 that vanishes on $B_{1/2}(0)\cong U$
 and is identically $1$ in a neighbourhood of $Q\setminus\Int V$.
 Set $\chi^p=\chi\circ\mu^p$ and
 \[
 \hat{\theta}^p=\rmd(\chi^pf^p)
 \]
 and observe that
 $\hat{\theta}^p|_{U^p}=0$.
 This defines a $1$-form $\hat{\theta}$
 on $\widetilde{Q}$
 that is equal to $\theta$ in the complement of the $V^p$'s
 and coincides with $\hat{\theta}^p$ on each $V^p$.
 By construction $\hat{\theta}$ is a primitive of
 $\mu^*\sigma$ that vanishes on $\mu^{-1}(U)$.
 
 It remains to show boundedness of $\hat{\theta}$
 on $(\widetilde{Q},\tilde{h})$.
 For this it will suffice to obtain a bound for
 \[
 \hat{\theta}^p=f^p\rmd\chi^p+\chi^p\theta^p
 \]
 independently of $p$.
 Of course $\chi^p$ is bounded by $1$.
 By chain rule we have
 \[
 \rmd\chi^p=\rmd\chi\circ T\mu^p\,.
 \]
 Because $\mu^p$ is an isometry
 we obtain a uniform bound on
 $|\rmd\chi^p|_{(\tilde{h})^{\flat}}$.
 Moreover,
 $|\theta^p|_{(\tilde{h})^{\flat}}$ can be estimated
 by the supremum of $|\theta|_{(\tilde{h})^{\flat}}$,
 which is bounded by assumption.
 Therefore,
 in order to obtain a uniform bound on
 $|\hat{\theta}^p|_{(\tilde{h})^{\flat}}$
 we need a uniform bound on $|f^p|$.
 
 For this recall the Poincar\'e--Lemma.
 Identify $D^n\cong V$ with $V^p$
 isometrically via $\mu^p$.
 In order to simplify the following computation
 in local coordinates
 we suppress the superscript $p$ from the notation.
 The $1$-form $\theta$,
 which got identified with $\theta^p$,
 is closed.
 Write
 $\theta_{\mathbf{x}}=\theta_j(\mathbf{x})\rmd x^j$
 using summation convention
 for $\mathbf{x}=(x^1,\ldots,x^n)$ in $D^n$.
 For $t\in[0,1]$ we get
 $\theta_{t\mathbf{x}}(\mathbf{x})=\theta_j(t\mathbf{x})x^j$
 so that a primitive of $\theta$
 is given by
 \[
 f(\mathbf{x})=\int_0^1\theta_{t\mathbf{x}}(\mathbf{x})\rmd t\,.
 \]
 Hence, by the mean value theorem there exists
 $t_0\in[0,1]$ such that
 \[
 |f(\mathbf{x})|\leq
 |\theta_{t_0\mathbf{x}}(\mathbf{x})|\leq
 \|\theta\|_{h_{t_0\mathbf{x}}}|\mathbf{x}|_{h_{t_0\mathbf{x}}}\,.
 \]
 Observe that the operator norm $\|\theta\|_h$
 equals $|\theta|_{(\tilde{h})^{\flat}}$ pointwise
 and is, therefore, uniformly bounded.
 Moreover,
 by compactness of $D^n$
 the restriction of the metric $h$ to $D^n$
 is uniformly equivalent to the Euclidean metric
 so that $|\mathbf{x}|_{h_{t_0\mathbf{x}}}$
 admits a uniform bound.
 Therefore, the same holds true for $|f(\mathbf{x})|$.
 Consequently,
 the perturbed primitive $\hat{\theta}$ of $\mu^*\sigma$
 is bounded.
 
 In order to finish the proof of the lemma
 we have to verify non-exactness of the
 odd-dimensional symplectic form
 of the resulting
 virtually contact structure if $\dim Q\geq3$
 and $\sigma$ is non-exact.
 But this follows exactly as for Proposition \ref{prop:boundgivescontact}.
\end{proof}

%%%%%%%%%%%%%%%%%%%%%%%%%%%%%%%%%%%

\subsection{Proof of Theorem \ref{thm:mainthm}\label{subsec:pfofthm}}

In view of Example \ref{ex:gromovsexample}
we choose a closed Riemannian manifold $(Q,h)$
that is not simply connected.
Moreover,
choose a closed non-exact $2$-form $\sigma$ on $Q$
whose lift to the universal cover has a bounded primitive.
By a use of a cut-off function $\chi$
as in the proof of Lemma \ref{lem:boundsomcont}
we can cut-off a local primitive $\theta_V$
of $\sigma|_V$ for an embedded closed disc $V$.
Setting $\sigma$ equal to $\rmd(\chi\theta_V)$
on $V$ this results into a new magnetic $2$-form
that vanishes somewhere.
Notice,
that the cohomology class of $\sigma$ is unchanged
and the lift of $\sigma$ still has a bounded primitive.
In this situation
Lemma \ref{lem:boundsomcont}
yields a somewhere contact
virtually contact structure
$\big(\pi\co M'\ra M,\alpha,\omega,g\big)$
with $\omega$ being non-exact if $\dim Q\geq3$
and with $M$ being not simply connected,
cf.\ Remark \ref{rem:topofpi}.
With these preliminaries
Theorem \ref{thm:mainthm} will be a consequence
of the following theorem if $n\geq3$.

\begin{prop}
\label{prop:mplusmandmplust}
 Let $\big(\pi\co M'\ra M,\alpha,\omega,g\big)$
 be a somewhere contact
 virtually contact structure with non-exact $\omega$
 and denote by $(T,\ker\alpha_T)$ a
 contact manifold.
 Assume that $M$ and $T$
 are of dimension $2n-1$.
 Then the connected sums
 $M\#M$ and $M\#T$
 admit somewhere contact
 virtually contact structures
 whose odd-dimensional symplectic forms are non-exact.
 Moreover,
 if $M$ and $T$ are not simply connected,
 then the belt spheres of the
 connected sums $M\#M$ and $M\#T$
 represent non-trivial elements in $\pi_{2n-2}$.
\end{prop}

\begin{proof}
 Denote by $x\in U$ the base point of $M$
 where $U$ is an open subset of $M$
 according to the definition of being somewhere contact,
 see Section \ref{subsec:definitions}.
 Performing a covering connected sum of
 $\big(\pi\co M'\ra M,\alpha,\omega,g\big)$
 with itself for any bijection $b$
 of the base point fibre $\pi^{-1}(x)$
 yields a virtually contact structure
 on $M\#M$,
 see Section \ref{subsec:covconsum}.
 In order to obtain a virtually contact structure
 on $M\#T$ consider the covering
 obtained by the disjoint union
 of $(T\times\{y\},\alpha_T)$, $y\in\pi^{-1}(x)$
 and perform covering connected sum.
 Non-exactness of the odd-dimensional symplectic form
 of the constructed virtually contact structures
 follows with Lemma \ref{lem:nontrivaftercovconsum}.
 Further,
 in both cases the resulting covering
 contact manifold admits a contact embedding
 of the upper boundary of a standard symplectic $1$-handle
 as it is discussed in Remark \ref{rem:conthandle}.
 In particular,
 the virtually contact structures
 on the surged manifolds are somewhere contact.
 Moreover,
 if $M$ and $T$ both are not simply connected,
 then the belt sphere represents a non-trivial homotopy class
 in $\pi_{2n-2}$ by 
 the proof of \cite[Proposition 3.10]{Hatcher}.
\end{proof}

This finishes the proof of Theorem \ref{thm:mainthm} if $n\geq3$.
The reason why the above argumentation does not work
for $n=2$ is that the odd-dimensional symplectic structure
obtained from a twisted cotangent bundle
of a surface $Q$ is necessarily exact if $Q$ is not a $2$-torus,
cf.\ the discussion on the end of the proof of
Proposition \ref{prop:boundgivescontact}.
In order to construct non-trivial
virtually contact structures
in dimension $3$
that are a non-trivial connected sum
we make the following observations:

\begin{prop}
\label{prop:pertofcontwnegcurv}
 Let $(M,\ker\alpha_M)$ be a closed connected contact manifold.
 Assume that $M$ carries a metric of negative sectional curvature
 and a non-exact closed $2$-form $\eta$.
 Then there exists a somewhere contact
 virtually contact structure
 $\big(\pi\co M'\ra M,\alpha,\omega,g\big)$
 on $M$
 such that $\omega$ is cohomologous to
 a positive multiple of $\eta$.
\end{prop}

\begin{proof}
 By using a suitable local cut-off of $\eta$
 we assume that there exists an open subset
 $V\subset M$ such that $\eta|_V=0$.
 This does not change the cohomology class of $\eta$.
 As explained in the proof of
 Lemma \ref{lem:boundsomcont}
 we can further assume that
 $\theta|_{\pi^{-1}(U)}=0$ for an open subset
 $U\subset\bar{U}\subset V$ of $M$.
 With \cite[0.2.A.]{gr91} $\pi^*\eta$
 has a bounded primitive $\theta$
 on the universal cover denoting by
 $\pi$ the corresponding covering map.
 For $\varepsilon>0$ sufficiently small
 the lift of the $2$-form
 $\omega=\rmd\alpha_M+\varepsilon\eta$
 along $\pi$ has a bounded primitive
 $\alpha=\pi^*\alpha_M+\varepsilon\theta$
 in the sense of \eqref{eq:gb1}
 that is a contact form.
 By shrinking $\varepsilon>0$ if necessary
 the contact form $\alpha$ satisfies \eqref{eq:gb2}
 as an argumentation by contradiction shows.
\end{proof}

Observe that $M$ is aspherical
in contrast to the examples given in
Proposition \ref{prop:mplusmandmplust}
and that by the theorem of Hadamard--Cartan
the compact manifold $M$
can not be simply connected.
Examples in dimension $3$ can be obtained as follows:

\begin{exwith}
 \label{ex:anosovthurston}
 Let $M$ be the mapping torus
 of a closed orientable surface of higher genus
 with monodromy diffeomorphism being
 pseudo-Anosov.
 By a theorem of Thurston \cite{thu88}
 $M$ is hyperbolic.
 Moreover, the Betti numbers $b_1=b_2$
 of $M$ are non-zero
 so that a non-exact closed $2$-form $\eta$
 can be found.
 By Martinet's theorem \cite[Theorem 4.1.1]{gei08}
 $M$ has a contact form $\alpha_M$.
\end{exwith}

A covering contact connected sum of
the somewhere contact virtually contact manifold
$M$ obtained with Example \ref{ex:anosovthurston}
and Proposition \ref{prop:pertofcontwnegcurv}
as described in Proposition \ref{prop:mplusmandmplust}
results in a non-trivial virtually contact manifold.
such that the related belt sphere
represents a non-trivial class in $\pi_{2n-2}$.
This finishes the proof of Theorem \ref{thm:mainthm}.
\hfill Q.E.D.

%%%%%%%%%%%%%%%%%%%%%%%%%%%%%%%%%%%

\subsection{Being prime\label{subsec:st*qisprim}}

Recall that a closed connected manifold $M$
is called {\bf prime}
if whenever written as a connected sum $M=M_1\#M_2$
one of the summands $M_1$ and $M_2$
is a homotopy sphere.
The connected sum with a homotopy sphere
is called to be {\bf trivial}.
We remark that the virtually contact manifolds
constructed in Section \ref{subsec:pfofthm}
are obtained by a non-trivial connected sum
and are, therefore, not prime.
This follows from the corresponding belt sphere
not to be contractible inside the surged manifold.

The aim of the following proposition is
to show that the examples
of virtually contact structures
given in Section \ref{subsec:pfofthm}
differ from the one obtained on
unit cotangent bundles $M\cong ST^*Q$
of $n$-dimensional
Riemannian manifolds of negative
sectional curvature studied
in Section \ref{subsec:boundprim}.
Recall, that by Hadamard--Cartan's theorem
the universal cover of a Riemannian manofold
of non-positive sectional curvature is
diffeomorphic to $\R^n$.

\begin{prop}
 The total space $ST^*Q$
 of the unit cotangent bundle
 of a closed connected
 aspherical $n$-dimensional manifold $Q$
 with respect to any metric on $Q$ is prime.
\end{prop}

\begin{proof}
 As $Q$ is aspherical by Whitehead's theorem
 the universal cover $\widetilde{Q}$
 of $Q$ contracts to its base point,
 see \cite[Theorem 4.5]{Hatcher}.
 Therefore, the cotangent bundle of $\widetilde{Q}$
 is trivial and $ST^*\widetilde{Q}$,
 which is diffeomorphic to $\widetilde{Q}\times S^{n-1}$,
 is homotopy equivalent to $S^{n-1}$.
 
 If $n=2$, then the universal cover of $ST^*Q$ is $\R^3$,
 see Remark \ref{rem:topofpi}.
 By Alexander's theorem $\R^3$ is {\bf irreducible},
 i.e.\ any embedded $2$-sphere bounds a ball,
 see \cite[Theorem 1.1]{Hatcher}.
 With \cite[Proposition 1.6]{Hatcher}
 the closed $3$-manifold
 $ST^*Q$ itself is irreducible and, therefore, prime.
 
 If $n\geq3$, then the universal cover of $ST^*Q$
 is diffeomorphic to $\widetilde{Q}\times S^{n-1}$.
 Consider an embedded $(2n-2)$-sphere $S_b$
 in $ST^*Q$ thinking of it as the belt sphere of a
 connected sum decomposition of $ST^*Q$.
 Let $\widetilde{S}_b$ be a lift of $S_b$
 to the universal cover of $ST^*Q$.
 Because the homology of the universal cover of $ST^*Q$
 vanishes in degree $2n-2$ any lift of $S_b$
 is the boundary of a bounded domain
 whose closure we denote by $\Omega_0$.
 We choose $\widetilde{S}_b$ so that $\Omega_0$
 does not contain any other of the lifts of $S_b$.
 The closure of the unbounded component of the complement
 of $\widetilde{S}_b$ is denoted by $\Omega_1$.
 Therefore, we obtain
 \[
 \widetilde{Q}\times S^{n-1}
 \cong
 \widetilde{ST^*Q}=
 \Omega_0\cup_{\widetilde{S}_b}\Omega_1
 \,.
 \]
 By Seifert--van Kampen's theorem $\Omega_0$
 must be simply connected.
 Moreover,
 the boundary operator of the Mayer--Vietoris sequence
 with respect to the above decomposition vanishes
 in all positive degrees.
 Indeed,
 we can take the image of $\{q\}\times S^{n-1}$,
 for $q\in\widetilde{Q}\simeq\{*\}$,
 as a generator of the homology in degree $n-1$
 so that its intersection with $\Omega_0$,
 and hence with $\widetilde{S}_b$, is empty.
 Therefore,
 the Mayer--Vietoris sequence
 reduces to the following short exact sequences
 \[
 0\ra
 H_k\widetilde{S}_b\ra
 H_k\Omega_0\oplus H_k\Omega_1\ra
 H_k\big(\widetilde{Q}\times S^{n-1}\big)\ra
 0
 \]
 for all positive $k$.
 This implies that $\Omega_0$
 has the homology of a ball.
 To see this for $k=n-1$ notice
 that the generator of the homology
 in degree $n-1$
 of the universal cover of $ST^*Q$
 is chosen to be contained in $\Omega_1$.
 The vanishing in degree $2n-2$
 follows with $\widetilde{S}_b\cong S^{2n-2}$
 being the boundary of $\Omega_0$.
 Therefore,
 $\Omega_0$ is a simply connected
 $(2n-1)$-dimensional homology ball
 with boundary $S^{2n-2}$.
 With \cite[p.~108, Proposition A and p.~110,
 Proposition C]{mil65} it follows that
 $\Omega_0$ is diffeomorphic
 to a $(2n-1)$-dimensional disc.
 With the arguments used in the proofs of
 \cite[Proposition 1.6 and Proposition 3.10]{Hatcher}
 this yields
 that $S_b$ bounds a $(2n-1)$-dimensional disc
 in $ST^*Q$ meaning that the assumed
 connected sum decomposition is trivial.
 After all, we see that $ST^*Q$ has to be prime.
\end{proof}

%%%%%%%%%%%%%%%%%%%%%%%%%%%%%%%%%%%
%%%%%%%%%%%%%%%%%%%%%%%%%%%%%%%%%%%

\section{Morse potentials\label{sec:morsepot}}

This section is devoted to a proof of Theorem \ref{thm:2ndthm}.

%%%%%%%%%%%%%%%%%%%%%%%%%%%%%%%%%%%

\subsection{Morsification\label{subsec:morsifi}}

We consider the Hamiltonian function
\[
H(u)=\frac12|u|^2_{h^{\flat}}+V\big(\tau(u)\big)
\]
of classical mechanics on $T^*Q$,
where $\tau\co T^*Q\ra Q$ is the cotangent bundle
and $(Q,h)$ is a closed oriented connected Riemannian manifold.
The linearization of $H$ at a point $u\in T^*Q$
can be written as
\[
T_uH=h^{\flat}\big(u,K_u(\,.\,)\big)+T_{\tau(u)}V\circ T_u\tau\,,
\]
where $K_u\co T_u(T^*Q)\ra T^*_{\tau(u)}Q$
is the connection operator of $h^{\flat}$.
In particular,
$u$ is a critical point of $H$
if and only if
$u$ is contained in the zero section $Q$ of $T^*Q$
and is a critical point of the potential $V\co Q\ra\R$.

This is of particular interest if $V$
is a Morse function what we will assume in the following.
Then $H$ will be a Morse function too.
This is because to the potential $V$
a positive definite quadratic form
with respect to the fibre direction is added.
In particular,
the Morse indices of a critical point are the same
for both functions $V$ and $H$.

%%%%%%%%%%%%%%%%%%%%%%%%%%%%%%%%%%%

\subsection{Topology of the energy surface\label{subsec:topofensur}}

We choose a Morse function $V$ on $Q$
that has a unique local maximum.
We assume that the maximum of $V$ is equal to $1$ and
that all critical points of index
less or equal than $n-1$ have critical value
smaller than $-1$.
For the regular value $0$ we consider the energy surface
$M=\{H=0\}$.

The sublevel set $W=\{H\leq0\}$
is a CW-complex of dimension
less or equal than $n-1$.
In particular,
$H_kW=0$ for all $k\geq n$
and $H_{n-1}W$ is torsion-free.
Hence, the boundary operator of the long exact sequence
of the pair $(W,M)$ induces an isomorphism
$H_{n+1}(W,M)\ra H_nM$.
Moreover, by the universal coefficient theorem
and Poincar\'e duality $H_{n-1}W$ injects into
$H_{n+1}(W,M)$ naturally.
In fact,
the Poincar\'e duality isomorphism
$H^{n-1}W\ra H_{n+1}(W,M)$
can be given in terms of the Morse functions
meaning that the classes in $H_{n+1}(W,M)$
can be represented by cocore discs $\{*\}\times D^{n+1}$,
see \cite[Remark on p.~35/36 and Theorem 7.5]{mil65}.
Therefore, the corresponding belt spheres $\{*\}\times S^n$
generate a free subgroup of $H_nM$
that is isomorphic to $H_{n-1}W$
as an application of the boundary operator shows.

The {\bf negative set} $N=\{V\leq0\}\subset Q$
is a deformation retract of $W$.
Hence, $H_{n-1}W$ and $H_{n-1}N$ are isomorphic.
By the assumptions on the Morse function $V$
we have $N\simeq Q\setminus\{*\}$
so that $H_{n-1}N=H_{n-1}Q$.
Therefore,
$H_{n-1}Q$ injects into $H_nM$
whose image is freely generated by belt spheres.
Denoting by $b_kQ$ the Betti numbers of $Q$
and using $b_1Q=b_{n-1}Q$
the Hurewicz homomorphism yields
\[
\pi_nM\geq\Z^{b_1Q}\,.
\]
This verifies the claim on the $n$-th homotopy group
in Theorem \ref{thm:2ndthm}.

\begin{exwith}
 If $Q$ is a closed Riemann surface of genus $g$,
 then $M$ is equal to the connected sum
 $S^3\#(2g)\big(S^1\times S^2\big)$.
\end{exwith}

%%%%%%%%%%%%%%%%%%%%%%%%%%%%%%%%%%%

\subsection{Virtually contact type\label{subsec:virtcontype}}

Let $\sigma$ be a $2$-form on $Q$
that vanishes on $\{V>-1\}$ and
consider the twisted symplectic form
$\omega_{\sigma}=\rmd\lambda+\tau^*\sigma$
on $T^*Q$.
Let $\theta$ be a bounded primitive of
$\mu^*\sigma$ denoting by
$\mu\co\widetilde{Q}\ra Q$
the universal covering.
By the proof of Lemma \ref{lem:boundsomcont}
we can assume that $\theta$ vanishes on
$\mu^{-1}\big(\{V>-1\}\big)$.

By multiplying $\sigma$ with a
small positive constant
we achieve that 
\[
\frac12|\theta|^2_{(\tilde{h})^{\flat}}<\frac12\,.
\]
This implies that $\widetilde{H}(\theta)$
is negative on $\mu^{-1}\big(\{V\leq-1\}\big)$.
Therefore,
as in the proof of Proposition \ref{prop:boundgivescontact},
\[
\big(\tilde{\lambda}+\tilde{\tau}^*\theta\big)|_{TM'}
\]
is a contact form on the intersection of $M'$ with
$\big(T^*\mu\big)^{-1}(T^*\{V\leq-1\})$
satisfying \eqref{eq:gb1} and \eqref{eq:gb2}.

Over the remaining part $U:=\{V>-1\}$
we perturb the Liouville form as follows:
Choose a function $F$ on $T^*Q$
whose support is contained in
$T^*U$ such that $(\lambda+\rmd F)(X_H)>0$
on $M\cap T^*U$,
see \cite[Lemma 5.2]{cfp10} or \cite[p. 137]{sz16}.
Therefore, $(\lambda+\rmd F)|_{TM}$
defines a contact form on $M\cap T^*U$.
Consequently,
\[
\alpha=\big(
\tilde{\lambda}+
\tilde{\tau}^*\theta+
\rmd\widetilde{F}\big)|_{TM'}
\]
is a contact form on $M'$,
where $\widetilde{F}=F\circ T^*\mu$.
Observe, that $\bar{U}$ is a compact set
and that the magnetic term $\sigma$
and the chosen primitive $\theta$
of the lift $\mu^*\sigma$
vanish over $\mu^{-1}(U)$.
Hence,
all involved differential forms are lifts
of differential forms that are
defined on a compact set.
In other words,
\eqref{eq:gb1} and \eqref{eq:gb2}
are satisfied along $M'\cap T^*\big(\mu^{-1}(U)\big)$
so that $\alpha$ defines a virtually contact structure.

\begin{rem}
 \label{mane}
 The Ma\~n\'e critical value of the described magnetic system
 equals $1$ as the maximum of $V$ is always a lower bound.
\end{rem}

%%%%%%%%%%%%%%%%%%%%%%%%%%%%%%%%%%%

\subsection{Exactness\label{subsec:exactness}}

The resulting odd-dimensional
symplectic form on $M$ is equal to
$\omega=(\rmd\lambda+\tau^*\sigma)|_{TM}$.
This form is exact
precisely if $\tau^*\sigma|_{TM}$ is exact,
which is the case
provided that $\sigma$ restricts to
an exact form on $\bar{N}$.
Invoking de Rham's theorem
and $N\simeq Q\setminus\{*\}$
we see that $\omega$ will be exact
in dimension $2n-1=3$.
If $n\geq3$ the exactness of $\tau^*\sigma|_{TM}$
is equivalent to the one of $\sigma$ on $Q$.
This follows with the Gysin sequence
for the unit cotangent bundle
of $\{V\leq-1\}$,
for which the map induced by $\tau$
is injective in degree $2$,
and an extension argument for primitive $1$-forms
over $U$, which is diffeomorphic to $D^n$.

In other words,
for $n=2$ the odd-dimensional symplectic form $\omega$
is always exact; for $n\geq3$ the odd-dimensional symplectic form
$\omega$ can be chosen to be non-exact precisely if
the Betti number $b_2Q$ does not vanish.

%%%%%%%%%%%%%%%%%%%%%%%%%%%%%%%%%%%

\subsection{Proof of Theorem \ref{thm:2ndthm}\label{subsec:pfofthm2}}

According to the construction
given in Sections \ref{subsec:topofensur},
\ref{subsec:virtcontype}, and \ref{subsec:exactness}
and Example \ref{ex:gromovsexample}
it suffices to find oriented closed manifolds $Q$
with non-trivial Betti numbers $b_1Q$ and $b_2Q$
that allow a Riemannian metric and
a closed non-exact $2$-form $\sigma$
such that the lift of $\sigma$ has a bounded primitive.

In dimension $n=2$ we can take any
closed oriented hyperbolic surface
and any $2$-form as magnetic term.
With Example \ref{ex:anosovthurston}
the case $n=3$ can be treated similarly.
In view of K\"unneth's formula
taking products in the sense of
Example \ref{ex:gromovsexample}
yields higher dimensional examples.
Because for any $b\in\N$
we find a manifold $Q$ with
the above listed properties
satisfying $b_1Q\geq b$
the claim of Theorem \ref{subsec:pfofthm2}
follows.
\hfill Q.E.D.

\begin{rem}
 For $b\geq 2$ the manifold $M$
 constructed in Section \ref{subsec:pfofthm2}
 is not diffeomorphic to a unit cotangent bundle
 of a closed aspherical manifold $Q$
 as such a $S^{n-1}$-bundle over $Q$
 has vanishing $\pi_2$ if $n=2$,
 $\pi_3$ equal to $\Z_2$ if $n=3$,
 and $\pi_n$ equal to $\Z$ if $n\geq 4$.
\end{rem}

%%%%%%%%%%%%%%%%%%%%%%%%%%%%%%%%%%%
%%%%%%%%%%%%%%%%%%%%%%%%%%%%%%%%%%%

\begin{ack}
  We would like to thank Peter Albers, Gabriele Benedetti,
  Youngjin Bae, Urs Frauenfelder, Stefan Friedl, Hansj\"org Geiges,
  Jarek K{\c{e}}dra, and Stefan Suhr.
\end{ack}

%%%%%%%%%%%%%%%%%%%%%%%%%%%%%%%%%%%
%%%%%%%%%%%%%%%%%%%%%%%%%%%%%%%%%%%

\end{document}